\documentclass[sn-mathphys,Numbered]{sn-jnl}

\usepackage{graphicx}%
\usepackage{multirow}%
\usepackage{amsmath,amssymb,amsfonts}%
\usepackage{amsthm}%
\usepackage{mathrsfs}%
\usepackage[title]{appendix}%
\usepackage{xcolor}%
\usepackage{textcomp}%
\usepackage{manyfoot}%
\usepackage{booktabs}%
\usepackage{algorithm}%
\usepackage{algorithmicx}%
\usepackage{algpseudocode}%
\usepackage{listings}%

\theoremstyle{article}%
\newtheorem{theorem}{Theorem}[section]
\newtheorem{proposition}[theorem]{Proposition}%

\theoremstyle{thmstyletwo}%

\theoremstyle{thmstylethree}%
\newtheorem{corollary}{Corollary}[section]
\newtheorem{lemma}{Lemma}[section]
\allowdisplaybreaks
\begin{document}

\title[On certain harmonic zeta functions]{On certain harmonic zeta functions}


\author[1]{\fnm{M\"{u}m\"{u}n} \sur{Can}}\email{mcan@akdeniz.edu.tr}

\author[1]{\fnm{Levent} \sur{Karg\i n}} \email{lkargin@akdeniz.edu.tr}

\author[1]{\fnm{Mehmet} \sur{Cenkci}} \email{cenkci@akdeniz.edu.tr}

\author[1]{\fnm{Ayhan} \sur{Dil}} \email{adil@akdeniz.edu.tr}

\affil[1]{\orgdiv{Department of Mathematics}, \orgname{Akdeniz University}, \orgaddress{\city{Antalya}, \postcode{07058}, \country{Turkey}}}


\abstract{This study deals with certain harmonic zeta functions, one of them occurs in
the study of the multiplication property of the harmonic Hurwitz zeta
function. The values at the negative even integers are found and Laurent
expansions at poles are described. Closed-form expressions are derived for the
Stieltjes constants that occur in Laurent expansions in a neighborhood of
$s=1$. Moreover, as a bonus, it is obtained that the values at the positive
odd integers of three harmonic zeta functions can be expressed in closed-form
evaluations in terms of zeta values and log-sine integrals.}




\maketitle

\section{Introduction}

The Hurwitz zeta function $\zeta\left(  s,a\right)  $ defined by the series%
\[
\zeta\left(  s,a\right)  =\sum_{n=0}^{\infty}\frac{1}{\left(  n+a\right)
^{s}},\text{ }\operatorname{Re}\left(  s\right)  >1,
\]
$a\in\mathbb{C}\setminus\left\{  0,-1,-2,\ldots\right\}  $ has a meromorphic
continuation to the complex $s$-plane with simple pole at $s=1$. It is well
known that around this simple pole, $\zeta\left(  s,a\right)  $ has the
following Laurent expansion:
\[
\zeta\left(  s,a\right)  =\frac{1}{s-1}+\sum_{m=0}^{\infty}\left(  -1\right)
^{m}\frac{\gamma\left(  m,a\right)  }{m!}\left(  s-1\right)  ^{m},
\]
where the coefficients $\gamma\left(  m,a\right)  $ are called generalized
Stieltjes constants. In particular, for $a=1$, $\zeta\left(  s,1\right)
=\zeta\left(  s\right)  $ is the Riemann zeta function and $\gamma\left(
m,1\right)  =\gamma\left(  m\right)  $ are the Stieltjes constants (see for
example \cite{BCh,LiT}). The constant $\gamma\left(  0\right)  $ is the
famous Euler-Mascheroni constant: $\gamma=\gamma\left(  0\right)=0.577\,215\,664\,9\ldots$. It is shown that these constants can be presented
by the limit%
\[
\gamma\left(  m,a\right)  =\lim_{N\rightarrow\infty}\left(  \sum_{n=1}%
^{N}\frac{\log^{m}\left(  n+a\right)  }{n+a}-\frac{\log^{m+1}\left(
N+a\right)  }{m+1}\right)
\]
(see \cite{Be}). There is comprehensive literature on deriving series and
integral representations, which usually allow a more accurate estimation, for
the Stieltjes constants and their extensions (see for example
\cite{Ad,Be,Bl1,Bl2,CDK,CDKCG,Che,Ch,Cof1,Cof2,C,FB,KDCC,SK}).

The harmonic zeta function, the Dirichlet series associated with the harmonic
numbers $H_{n}=1+1/2+\cdots+1/n$, is defined by
\[
\zeta_{H}\left(  s\right)  =\sum_{k=1}^{\infty}\frac{H_{k}}{k^{s}},\text{
}\operatorname{Re}\left(  s\right)  >1,
\]
and subject to many studies. Euler \cite[pp. 217--264]{E} has shown that
special values of the harmonic zeta function have relations to those of the
Riemann zeta values. Apostol and Vu \cite{AV}, and Matsuoka \cite{Ma} have
shown that the function $\zeta_{H}$ can be holomorphically continued to the
whole $s$-plane except for the poles at $s=1$, $s=0$ and $s=1-2j,$
$j\in\mathbb{N}$. Later, Boyadzhiev et al. \cite{BGP1} and Candelpergher and
Coppo \cite{CC} deal with the Laurent expansion of the form\textbf{ }%
\begin{equation}
\zeta_{H}\left(  s\right)  =\frac{a_{-1}}{s-b}+a_{0}+O\left(  s-b\right)
,\label{zhL}%
\end{equation}
and give explicitly $a_{0}$ when $b=0$ and $b=1-2j$, $j\in\mathbb{N}$.
Besides, in \cite{CC} and \cite{CDK}, it has been recorded that the harmonic
Stieltjes constants $\gamma_{H}\left(  m\right)  $ can be expressed as%
\[
\gamma_{H}\left(  m\right)  =\lim_{x\rightarrow\infty}\left(  \sum_{n\leq
x}\frac{H_{n}\ln^{m}n}{n}-\frac{\ln^{m+2}x}{m+2}-\gamma\frac{\ln^{m+1}x}%
{m+1}\right)
\]
and $\gamma_{H}\left(  0\right)  =\left(  \zeta\left(  2\right)  +\gamma
^{2}\right)  /2$, where $\gamma_{H}\left(  m\right)  $ are defined by
\[
\zeta_{H}\left(  s\right)  =\frac{1}{\left(  s-1\right)  ^{2}}+\frac{\gamma
}{s-1}+\sum_{m=0}^{\infty}\left(  -1\right)  ^{m}\frac{\gamma_{H}\left(
m\right)  }{m!}\left(  s-1\right)  ^{m},\text{ }0<\left\vert s-1\right\vert
<1.
\]
It is worth mentioning that, in Alkan \cite{Al}, certain real numbers and
log-sine integrals have been shown as a combination of special values of the
harmonic zeta function and the Riemann zeta function, which provided strong
approximations for them. Recently, Alzer and Choi \cite{AC} have introduced
four types of parametric Euler sums, namely%
\begin{align*}
S_{z,s}^{++}\left(  a,b\right)   &  =\sum_{n=1}^{\infty}\frac{\mathcal{H}%
_{n}^{\left(  z\right)  }\left(  a\right)  }{\left(  n+b\right)  ^{s}},\text{
}S_{z,s}^{+-}\left(  a,b\right)  =\sum_{n=1}^{\infty}\left(  -1\right)
^{n+1}\frac{\mathcal{H}_{n}^{\left(  z\right)  }\left(  a\right)  }{\left(
n+b\right)  ^{s}},\\
S_{z,s}^{-+}\left(  a,b\right)   &  =\sum_{n=1}^{\infty}\frac{\mathcal{A}%
_{n}^{\left(  z\right)  }\left(  a\right)  }{\left(  n+b\right)  ^{s}},\text{
}S_{z,s}^{--}\left(  a,b\right)  =\sum_{n=1}^{\infty}\left(  -1\right)
^{n+1}\frac{\mathcal{A}_{n}^{\left(  z\right)  }\left(  a\right)  }{\left(
n+b\right)  ^{s}},
\end{align*}
where%
\[
\mathcal{H}_{n}^{\left(  z\right)  }\left(  a\right)  =\sum_{k=1}^{n}\frac
{1}{\left(  k+a\right)  ^{z}}\text{ and }\mathcal{A}_{n}^{\left(  z\right)
}\left(  a\right)  =\sum_{k=1}^{n}\frac{\left(  -1\right)  ^{k-1}}{\left(
k+a\right)  ^{z}},
\]
$a,b\in\mathbb{C}\backslash\left\{  -1,-2,-3,\ldots\right\}  $ and $s,$
$z\in\mathbb{C}$ are adjusted so that the involved defining series converge.
They have investigated analytic continuations via summation formulas and given
shuffle relations. Very recently, in \cite{KDCC}, we have studied the
harmonic Hurwitz zeta function
\begin{equation}
\zeta_{H}\left(  s,a\right)  :=S_{1,s}^{++}\left(  a-1,a-1\right)  =\sum
_{n=0}^{\infty}\frac{H_{n}\left(  a\right)  }{\left(  n+a\right)  ^{s}%
},\label{hhz}%
\end{equation}
where $a\in\mathbb{C}\backslash\left\{  0,-1,-2,-3,\ldots\right\}  $,
$\operatorname{Re}(s)>1$ and $H_{n}\left(  a\right)  =\sum_{k=0}^{n}1/(k+a)$.
Among others, they have investigated analytic continuation via contour
integral and introduced a limit presentation for the constants $\gamma
_{H}\left(  m,a\right)  $ defined by
\begin{equation}
\zeta_{H}\left(  s,a\right)  =\frac{1}{\left(  s-1\right)  ^{2}}+\frac
{\gamma\left(  0,a\right)  }{s-1}+\sum\limits_{n=0}^{\infty}\left(  -1\right)
^{n}\frac{\gamma_{H}\left(  n,a\right)  }{n!}\left(  s-1\right)  ^{n},\text{
}0<\left\vert s-1\right\vert <1\label{hzetaL}%
\end{equation}
and closed-form expressions for $\gamma_{H}\left(  m,1\right)  $ and
$\gamma_{H}\left(  m,1/2\right)  $ in terms of certain constants and integrals.

We note that $H_{n}=H_{n-1}\left(  1\right)  =\mathcal{H}_{n}^{\left(
1\right)  }\left(  0\right)  $, $\mathcal{H}_{n}^{\left(  1\right)  }\left(
-\frac{1}{2}\right)  =2O\left(  n\right)  =2\sum_{k=1}^{n}\frac{1}{2k-1}$,
$\mathcal{A}_{n}^{\left(  1\right)  }\left(  0\right)  =H_{n}^{-}=\sum
_{k=1}^{n}\frac{\left(  -1\right)  ^{k-1}}{k}$, the skew harmonic numbers, and
$H_{n-1}\left(  a\right)  =\mathcal{H}_{n}^{\left(  1\right)  }\left(
a-1\right)  $. It is obvious that $S_{1,s}^{++}\left(  0,0\right)  $ is the
harmonic zeta function $\zeta_{H}\left(  s\right)  $, and refer to
\cite{Ba,BoBB,CKDS,EL,J,KCDC,LiQ,NKS,QSL,So,SoC,WX} for studies on several types of Euler sums. In particular, analytic continuations of $S_{1,s}^{+-}\left(
0,0\right)  =\eta_{H}\left(  s\right)  $, $S_{1,s}^{--}\left(  0,0\right)
=\eta_{H^{-}}\left(  s\right)  $ and $S_{1,s}^{-+}\left(  0,0\right)
=\zeta_{H^{-}}\left(  s\right)  $ have been investigated by Boyadzhiev et al.
in \cite{BGP2}. It is known that the values $\zeta_{H}\left(  m\right)  $ (cf.
\cite{E}), $\eta_{H}\left(  2m\right)  $, $\eta_{H^{-}}\left(  2m\right)  $
(cf. \cite{Si} or \cite[Theorem 7.1]{FlS}), $S_{1,m}^{++}\left(  0,\frac
{1}{2}\right)  $ (cf. \cite[Theorem 1]{SoCv}), $\zeta_{O}\left(  2m\right)  $
and $\zeta_{H}\left(  2m,1/2\right)  $ (cf. \cite[Eqs. (8a) and (9a)]{J}) have
closed-form evaluation formulas in terms of zeta values. In contrast to that
for integers $m\geq1$, the values $\eta_{H}\left(  2m+1\right)  $,
$\eta_{H^{-}}\left(  2m+1\right)  $, $\zeta_{O}\left(  2m+1\right)  $ and
$\zeta_{H}\left(  2m+1,1/2\right)  $ do not directly admit any closed-form
evaluations in terms of other well-known constants. With this one, Alkan
\cite{Al} shows that $\zeta_{O}\left(  m\right)  $ can always be written as a
combination of log-sine integrals over $\left[  0,2\pi\right]  $ and $\left[
0,\pi\right]  $:
\begin{equation}
\hspace{-0.13in}%
\begin{array}
[c]{l}%
\zeta_{O}\left(  m\right)  =\mathbb{Q}\pi^{m+1}+\sum\limits_{v=1}^{\left(
m-1\right)  /2}\mathbb{Q}\left(  \pi\right)  J_{2v+1,2\pi}+\sum\limits_{v=1}%
^{m-1}\mathbb{Q}\left(  \pi\right)  J_{v,\pi}\text{, }m\in2\mathbb{N}%
+1\text{,}\\
\zeta_{O}\left(  m\right)  =\mathbb{Q}\pi^{m+1}+\sum\limits_{v=1}^{\left(
m-1\right)  /2}\mathbb{Q}\left(  \pi\right)  I_{2v+1,2\pi}+\sum\limits_{v=1}%
^{m-1}\mathbb{Q}\left(  \pi\right)  I_{v,\pi}\text{, }m\in2\mathbb{N}\text{,}%
\end{array}
\label{ozetam}%
\end{equation}
where $\mathbb{Q}\left(  \pi\right)  $ denotes the totally transcendental
extension obtained by adjoining $\pi$ to $\mathbb{Q}$, and%
\[
I_{m,x}=\int\limits_{0}^{x}t^{m}\log\left(  2\sin\frac{t}{2}\right)  dt\text{
and }J_{m,x}=\int\limits_{0}^{x}t^{m}\log^{2}\left(  2\sin\frac{t}{2}\right)
dt.
\]

\subsection{Outline of the study}

In this study, we first consider the harmonic-type zeta function%
\[
\zeta_{A\left(  k\right)  }\left(  s\right)  =\sum\limits_{n=1}^{\infty}%
\frac{A_{n}\left(  k\right)  }{n^{s}}\text{,}%
\]
where the numbers $A_{n}\left(  k\right)  $ are defined by the generating
function%
\begin{equation}
-\frac{\ln\left(  1-x\right)  }{1-x^{k}}=\sum\limits_{n=1}^{\infty}%
A_{n}\left(  k\right)  x^{n}. \label{gAnk}%
\end{equation}
The function $\zeta_{A\left(  k\right)  }$ occurs in the study of the
multiplication property of the harmonic Hurwitz zeta function $\zeta
_{H}\left(  s,a\right)  $:

\begin{proposition}
\label{prop-raabe}For $k\in\mathbb{N}$, we have%
\begin{equation}
\sum\limits_{a=1}^{k}\zeta_{H}\left(  s,\frac{a}{k}\right)  =k^{s+1}%
\zeta_{A\left(  k\right)  }\left(  s\right)  . \label{Raabe}%
\end{equation}

\end{proposition}

Using contour integral we prove the following theorem, which shows that
$\zeta_{A\left(  k\right)  }\left(  s\right)  $ has an analytic continuation
to the whole complex $s$-plane except for the poles $s=1$, $s=0$ and $s=1-2j,$
$j\in\mathbb{N}$.

\begin{theorem}
\label{mteo1}We have%
\begin{equation}
\zeta_{A\left(  k\right)  }\left(  s\right)  =\Gamma\left(  1-s\right)
I\left(  s,k\right)  +\zeta\left(  s+1\right)  -\frac{1}{k^{s}}\left(
\psi\left(  s\right)  \zeta\left(  s\right)  +\zeta^{\prime}\left(  s\right)
-\zeta\left(  s\right)  \log k\right)  . \label{2}%
\end{equation}
where $I\left(  s,k\right)  $ is defined by (\ref{Isk}).
\end{theorem}

The proofs of Proposition \ref{prop-raabe} and Theorem \ref{mteo1} are given
in Section \ref{sec2}. With the help of (\ref{2}), in Section \ref{sec3}, we
determine the values $\zeta_{A\left(  k\right)  }\left(  -2n\right)  $ and
describe the Laurent expansion of $\zeta_{A\left(  k\right)  }\left(
s\right)  $ at the poles $s=1$, $s=0$ and $s=1-2j,$ $j\in\mathbb{N}$. In
addition, for $\gamma_{A\left(  k\right)  }\left(  m\right)  $ (the Laurent
coefficients of $\zeta_{A\left(  k\right)  }\left(  s\right)  $ in a
neighborhood of $s=1$) we give the closed-form expression (\ref{ev1}) and the
limit representation (\ref{lim}), by modifying the method of Briggs and
Buschman \cite{BB}.

Secondly, in Section \ref{sec5}, we offer two closed-form evaluation formulas for
$\gamma_{H}\left(  m,1/2\right)  $, the first one depends on (\ref{ev1}) and
the second one depends on the relationship between $\zeta_{H}\left(
s,1/2\right)  $, $\zeta_{H}\left(  s\right)  $ and $\eta_{H^{-}}\left(
s\right)  $. These two formulas provide a contribution to the evaluation of
$\gamma_{H^{-}}\left(  m\right)  $, the Laurent coefficient of $\zeta_{H^{-}%
}\left(  s\right)  $ in a neighborhood of $s=1$.

Thirdly, we express $\zeta_{O}\left(  s\right)  $ and $S_{1,s}^{++}\left(
0,1/2\right)  $ in terms of $\zeta_{H}\left(  s\right)  $, $\eta_{H}\left(
s\right)  $ and $\zeta_{H}\left(  s,1/2\right)  $. More precisely, in Section
\ref{sec6}, we prove the following.

\begin{theorem}
\label{mteo3}We have%
\begin{equation}
S_{1,s}^{++}\left(  0,\frac{1}{2}\right)  =2^{s}\left(  \eta_{H}\left(
s\right)  +\zeta_{H}\left(  s\right)  \right)  -\zeta_{H}\left(  s,\frac{1}%
{2}\right)  \label{19}%
\end{equation}
and%
\begin{equation}
\zeta_{O}\left(  s\right)  =\frac{1}{2}\left(  2^{s}-1\right)  \zeta
_{H}\left(  s\right)  -2^{s-1}\eta_{H}\left(  s\right)  . \label{ozeta}%
\end{equation}

\end{theorem}

Equations (\ref{19}) and (\ref{ozeta}) lead closed-form expressions for
$d_{m}$ and $\gamma_{O}\left(  m\right)  $, the Laurent coefficients of
$S_{1,s}^{++}\left(  0,1/2\right)  $ and $\zeta_{O}\left(  s\right)  $ in a
neighborhood of $s=1$, respectively. Besides, we reach at the values
$S_{1,-2m}^{++}\left(  0,1/2\right)  $ and $\zeta_{O}\left(  -2m\right)  $,
where $m\in\mathbb{N}\cup\left\{  0\right\}  $. As a bonus, the equations
(\ref{19}) and (\ref{ozeta}) with (\ref{3}) allow us to write the values
$\eta_{H}\left(  m\right)  $, $\eta_{H^{-}}\left(  m\right)  $ and $\zeta
_{H}\left(  m,1/2\right)  $ as%
\begin{align}
2^{m-1}\eta_{H}\left(  m\right)   &  =\frac{1}{2}\left(  2^{m}-1\right)
\zeta_{H}\left(  m\right)  -\zeta_{O}\left(  m\right)  ,\label{20a}\\
\zeta_{H}\left(  m,\frac{1}{2}\right)   &  =\left(  2^{m+1}-1\right)
\zeta_{H}\left(  m\right)  -2\zeta_{O}\left(  m\right)  -S_{1,m}^{++}\left(
0,\frac{1}{2}\right)  ,\label{20b}\\
2^{m}\eta_{H^{-}}\left(  m\right)   &  =2^{m}\zeta_{H}\left(  m\right)
-2\zeta_{O}\left(  m\right)  -S_{1,m}^{++}\left(  0,\frac{1}{2}\right)
\label{20c}%
\end{align}
for $2\leq m\in\mathbb{N}$. Thus, (\ref{20a})--(\ref{20c}) and (\ref{ozetam})
show that the values $\eta_{H}\left(  m\right)  $, $\eta_{H^{-}}\left(
m\right)  $ and $\zeta_{H}\left(  m,1/2\right)  $ can also be expressed in
closed-form evaluations in terms of zeta values and log-sine integrals over
$\left[  0,2\pi\right]  $ and $\left[  0,\pi\right]  $.

As a demonstration of the formulas (\ref{ev1}), (\ref{ev2a}) or (\ref{ev2b}),
(\ref{8b}), (\ref{gammao}) and (\ref{dn}) we present the first few values of
the corresponding Stieltjes constants:
\medskip
\begin{center}
$%
\begin{tabular}
[c]{lcc}\hline
$\gamma_{A\left(  k\right)  }\left(  n\right)  $ & $k=2$ & $k=3$%
\begin{tabular}
[c]{l}%
\smallskip\\
\end{tabular}
\\\hline
$n=0$ & \multicolumn{1}{l}{$1.02587476785559324\ldots$} &
\multicolumn{1}{l}{$1.06996099339579407$\ldots$%
\begin{tabular}
[c]{l}%
\smallskip\\
\end{tabular}
\ $}\\
$n=1$ & \multicolumn{1}{l}{$0.35594914036775009$\ldots} &
\multicolumn{1}{l}{$0.33314014622807110$\ldots}\\
$n=2$ & \multicolumn{1}{l}{$0.93150345624786643$\ldots} &
\multicolumn{1}{l}{$0.88997380163684287$\ldots$%
\begin{tabular}
[c]{l}%
\smallskip\\
\end{tabular}
\ $}\\
$n=3$ & \multicolumn{1}{l}{$2.93741964148769000$\ldots} &
\multicolumn{1}{l}{$2.88842478274497700$\ldots}\\
$n=4$ & \multicolumn{1}{l}{$11.8842517416686066$\ldots} &
\multicolumn{1}{l}{$11.8175214586523864$\ldots$%
\begin{tabular}
[c]{l}%
\smallskip\\
\end{tabular}
\ $}\\
$n=5$ & \multicolumn{1}{l}{$59.6981596440678285$\ldots} &
\multicolumn{1}{l}{$59.5558093930507495$\ldots}\\
& $k=4$ & $k=5$%
\begin{tabular}
[c]{l}%
\smallskip\\
\end{tabular}
\\\hline
$n=0$ & \multicolumn{1}{l}{$1.10841178277676204$\ldots} &
\multicolumn{1}{l}{$1.14109065573389757$\ldots%
\begin{tabular}
[c]{l}%
\smallskip\\
\end{tabular}
}\\
$n=1$ & \multicolumn{1}{l}{$0.32209830233514756$\ldots} &
\multicolumn{1}{l}{$0.31774187945874161$\ldots}\\
$n=2$ & \multicolumn{1}{l}{$0.85300634735541063$\ldots} &
\multicolumn{1}{l}{$0.82164783744427303$\ldots$%
\begin{tabular}
[c]{l}%
\smallskip\\
\end{tabular}
\ $}\\
$n=3$ & \multicolumn{1}{l}{$2.83187968374836900$\ldots} &
\multicolumn{1}{l}{$2.77288405256252639$\ldots}\\
$n=4$ & \multicolumn{1}{l}{$11.7351946009035882$\ldots} &
\multicolumn{1}{l}{$11.6393361732958488$\ldots%
\begin{tabular}
[c]{l}%
\smallskip\\
\end{tabular}
}\\
$n=5$ & \multicolumn{1}{l}{$59.4016636386555705$\ldots} &
\multicolumn{1}{l}{$59.2271261821701563$\ldots}%
\end{tabular}
\ \ $\medskip

$%
\begin{tabular}
[c]{lcc}\hline
& $\gamma_{H}\left(  n,1/2\right)  $ & $\gamma_{H^{^{-}}}\left(  n\right)  $%
\begin{tabular}
[c]{l}%
\smallskip\\
\end{tabular}
\\\hline
$n=0$
\begin{tabular}
[c]{l}%
\smallskip\\
\end{tabular}
& \multicolumn{1}{l}{$+4.395086911415665\ldots$} &
\multicolumn{1}{l}{$+0.42762775101889\ldots$
\begin{tabular}
[c]{l}%
\smallskip\\
\end{tabular}
}\\
$n=1$ & \multicolumn{1}{l}{$-2.209626703485579\ldots$} &
\multicolumn{1}{l}{$-0.16151097065250\ldots$}\\
$n=2$ & \multicolumn{1}{l}{$+2.919070396332389\ldots$} &
\multicolumn{1}{l}{$-0.05863891452811\ldots$%
\begin{tabular}
[c]{l}%
\smallskip\\
\end{tabular}
}\\
$n=3$ & \multicolumn{1}{l}{$+1.632080580126146\ldots$} &
\multicolumn{1}{l}{$-0.02184763116814\ldots$}\\
$n=4$ & \multicolumn{1}{l}{$+12.85545704209538\ldots$} &
\multicolumn{1}{l}{$-0.00679716118626\ldots$%
\begin{tabular}
[c]{l}%
\smallskip\\
\end{tabular}
}\\
$n=5$ & \multicolumn{1}{l}{$+59.20224186218603\ldots$} &
\multicolumn{1}{l}{$-0.00072943328309\ldots$}\\
$n=6$ & \multicolumn{1}{l}{$+359.9761038476681\ldots$} &
\multicolumn{1}{l}{$+0.00089596703007\ldots$%
\begin{tabular}
[c]{l}%
\smallskip\\
\end{tabular}
}\\
$n=7$ & \multicolumn{1}{l}{$+2518.053983773233\ldots$} &
\multicolumn{1}{l}{$+0.00006691948659\ldots$}\\
$n=8$ & \multicolumn{1}{l}{$+20153.65275669922\ldots$} &
\multicolumn{1}{l}{$-0.00201181883640\ldots$%
\begin{tabular}
[c]{l}%
\medskip\\
\end{tabular}
}\\
$n=9$ & \multicolumn{1}{l}{$+181410.3238030236\ldots$} &
\multicolumn{1}{l}{$-0.00437336373504\ldots$}\\\hline
& $\gamma_{O}\left(  n\right)  $ & $d_{n}$%
\begin{tabular}
[c]{l}%
\smallskip\\
\end{tabular}
\\\hline
$n=0$
\begin{tabular}
[c]{l}%
\smallskip\\
\end{tabular}
& \multicolumn{1}{l}{$+0.55260938885958\ldots$} &
\multicolumn{1}{l}{$+0.02814996749151\ldots$
\begin{tabular}
[c]{l}%
\smallskip\\
\end{tabular}
}\\
$n=1$ & \multicolumn{1}{l}{$-0.47923813171300\ldots$} &
\multicolumn{1}{l}{$+0.85147471204914\ldots$}\\
$n=2$ & \multicolumn{1}{l}{$+0.56056563855360\ldots$} &
\multicolumn{1}{l}{$-1.11455232720134\ldots$%
\begin{tabular}
[c]{l}%
\smallskip\\
\end{tabular}
}\\
$n=3$ & \multicolumn{1}{l}{$-0.56582227536904\ldots$} &
\multicolumn{1}{l}{$+1.16603391454342\ldots$}\\
$n=4$ & \multicolumn{1}{l}{$+0.53307981193217\ldots$} &
\multicolumn{1}{l}{$-1.00064761848714\ldots$%
\begin{tabular}
[c]{l}%
\smallskip\\
\end{tabular}
}\\
$n=5$ & \multicolumn{1}{l}{$-0.31150731317173\ldots$} &
\multicolumn{1}{l}{$+0.77991738610518\ldots$}\\
$n=6$ & \multicolumn{1}{l}{$+0.14509282444820\ldots$} &
\multicolumn{1}{l}{$+0.17758339430635\ldots$%
\begin{tabular}
[c]{l}%
\smallskip\\
\end{tabular}
}\\
$n=7$ & \multicolumn{1}{l}{$+0.86693551972323\ldots$} &
\multicolumn{1}{l}{$-0.09508032673120\ldots$}\\
$n=8$ & \multicolumn{1}{l}{$+0.61277165952460\ldots$} &
\multicolumn{1}{l}{$+5.33425919693367\ldots$%
\begin{tabular}
[c]{l}%
\smallskip\\
\end{tabular}
}\\
$n=9$ & \multicolumn{1}{l}{$+8.81674734666844\ldots$} &
\multicolumn{1}{l}{$+11.8927079050673\ldots$}%
\end{tabular}
\ \ $\medskip
\end{center}

\noindent(see \cite[p. 3 and p. 15]{KDCC} for the numerical values of
$\gamma_{H}\left(  n\right)  $ and $\eta_{H}^{\left(  v\right)  }\left(
1\right)  =\gamma_{\widetilde{H}}\left(  n\right)  $ used in
calculations.)$\ $

Throughout this study, a sum that is empty is regarded as zero, and thus
$H_{0}=0$.

\section{Proofs of Proposition \ref{prop-raabe} and Theorem \ref{mteo1}%
\label{sec2}}

\subsection{Proof of Proposition \ref{prop-raabe}}
We make use of
\[
\zeta_{H}\left(  s,a\right)  =\frac{1}{\Gamma\left(  s\right)  }\int%
_{0}^{\infty}\frac{x^{s-1}}{1-e^{-x}}e^{-xa}\Phi\left(  e^{-x},1;a\right)  dx
\]
\cite[Eq. (14)]{KDCC} to deduce that%
\begin{align*}
\sum\limits_{a=1}^{k}\zeta_{H}\left(  s,\frac{a}{k}\right)   &  =\frac
{1}{\Gamma\left(  s\right)  }\sum\limits_{a=1}^{k}\int_{0}^{\infty}%
\frac{x^{s-1}}{1-e^{-x}}e^{-x\frac{a}{k}}\Phi\left(  e^{-x},1;\frac{a}%
{k}\right)  dx\\
&  =\frac{k}{\Gamma\left(  s\right)  }\int_{0}^{\infty}\frac{x^{s-1}}%
{1-e^{-x}}\sum\limits_{n=0}^{\infty}\sum\limits_{a=1}^{k}\frac{e^{-\left(
kn+a\right)  x/k}}{kn+a}dx\\
&  \overset{a+nk=j}{=}\frac{k}{\Gamma\left(  s\right)  }\int_{0}^{\infty}%
\frac{x^{s-1}}{1-e^{-x}}\sum\limits_{j=1}^{\infty}\frac{e^{-j\left(
x/k\right)  }}{j}dx\\
&  =-\frac{k^{s+1}}{\Gamma\left(  s\right)  }\int_{0}^{\infty}t^{s-1}\frac
{\ln\left(  1-e^{-t}\right)  }{1-e^{-kt}}dt.
\end{align*}
We now exploit (\ref{gAnk}) and see that\
\begin{align*}
\sum\limits_{a=1}^{k}\zeta_{H}\left(  s,\frac{a}{k}\right)   &  =\frac
{k^{s+1}}{\Gamma\left(  s\right)  }\sum\limits_{n=1}^{\infty}A_{n}\left(
k\right)  \int_{0}^{\infty}t^{s-1}e^{-nt}dt\\
&  =k^{s+1}\sum\limits_{n=1}^{\infty}\frac{A_{n}\left(  k\right)  }{n^{s}}.
\end{align*}

\subsection{Proof of Theorem \ref{mteo1}}
Let%
\begin{equation}
I\left(  s,k\right)  =\frac{1}{2\pi i}\int_{C}\frac{z^{s-1}e^{kz}}{e^{kz}%
-1}\log\left(  \frac{e^{z}-1}{z}\right)  dz. \label{Isk}%
\end{equation}
Here $C$ denotes the Hankel contour which starts from $-\infty$ along the
lower side of the negative real axis, encircles the origin once in the
positive (counter-clockwise) direction and then returns to $-\infty$ along the
upper side of the negative real axis. The loop $C$ consists of the parts
$C=C_{-}\cup C_{+}\cup C_{\varepsilon}$, where $C_{\varepsilon}$ is a
positively-oriented circle of radius $\varepsilon$ about the origin, and
$C_{-}$ and $C_{+}$ are the lower and upper edges of a cut in the complex
$z$-plane along the negative real axis.

We focus on the integral $I\left(
s,k\right)  .$ Since $z=xe^{-\pi i}$ on $C_{-}$, $z=xe^{\pi i}$ on $C_{+}$ and
$z=\varepsilon e^{i\theta}$ on $C_{\varepsilon}$ we find that%
\begin{align*}
I\left(  s,k\right)   &  =\frac{1}{2\pi i}\left(  \int_{C_{-}}+\int_{C_{+}%
}+\int_{C_{\varepsilon}}\right)  \frac{z^{s-1}e^{kz}}{e^{kz}-1}\log\left(
\frac{e^{z}-1}{z}\right)  dz\\
&  =\frac{\sin\left(  s\pi\right)  }{\pi}\int_{\varepsilon}^{\infty}%
\frac{x^{s-1}e^{-kx}\log\left(  \frac{1-e^{-x}}{x}\right)  }{e^{-kx}%
-1}dx+\frac{1}{2\pi i}\int_{C_{\varepsilon}}\frac{z^{s-1}e^{kz}\log\left(
\frac{e^{z}-1}{z}\right)  }{e^{kz}-1}dz.
\end{align*}
For $\operatorname{Re}\left(  s\right)  >1$ the function $\frac{z^{s-1}e^{kz}%
}{e^{kz}-1}\log\left(  \frac{e^{z}-1}{z}\right)  $\ is analytic in a
neighborhood of $z=0$. So the integral over $C_{\varepsilon}$ vanishes.
Letting $\varepsilon\rightarrow0$ we reach at%
\begin{align}
I\left(  s,k\right)   &  =\frac{\sin\left(  s\pi\right)  }{\pi}\int%
_{0}^{\infty}\frac{x^{s-1}e^{-kx}}{e^{-kx}-1}\log\left(  \frac{1-e^{-x}}%
{x}\right)  dx\label{10}\\
&  =\frac{\sin\left(  s\pi\right)  }{\pi}\left(  \int_{0}^{\infty}%
\frac{x^{s-1}e^{-kx}}{e^{-kx}-1}\log\left(  1-e^{-x}\right)  dx+\int%
_{0}^{\infty}\frac{x^{s-1}e^{-kx}}{1-e^{-kx}}\log xdx\right)  .\nonumber
\end{align}

To calculate the first integral on the RHS we use (\ref{gAnk}):%
\[
\int_{0}^{\infty}x^{s-1}\frac{e^{-kx}\log\left(  1-e^{-x}\right)  }{e^{-xk}%
-1}dx=\Gamma\left(  s\right)  \sum_{n=k+1}^{\infty}\frac{A_{n-k}\left(
k\right)  }{n^{s}}.
\]

It can be seen from (\ref{gAnk}) that%
\begin{align*}
\sum\limits_{n=1}^{\infty}A_{n}\left(  k\right)  x^{n}  &  =-\frac{\ln\left(
1-x\right)  }{1-x^{k}}\\
&  =\sum\limits_{\substack{j\geq0,n\geq1\\n+kj=m}}^{\infty}\frac{x^{m}}%
{m-kj}=\sum\limits_{m=1}^{\infty}\sum\limits_{j=0}^{\left\lfloor \frac{m-1}%
{k}\right\rfloor }\frac{1}{m-kj}x^{m},
\end{align*}
i.e.
\[
A_{n}\left(  k\right)  =\sum\limits_{j=0}^{\left\lfloor \left(  n-1\right)
/k\right\rfloor }\frac{1}{n-kj}=\sum\limits_{\substack{v=1\\v\equiv n\left(
\operatorname{mod}k\right)  }}^{n}\frac{1}{v},
\]
where $\left\lfloor x\right\rfloor $ stands for the integer part of
$x\in\mathbb{R}$. Moreover the numbers $A_{n}\left(  k\right)  $ may be
presented as
\begin{equation}
A_{n}\left(  k\right)  =\sum\limits_{j=1}^{n}\frac{f_{n}\left(  j\right)  }%
{j}, \label{Ank}%
\end{equation}
where
\[
f_{n}\left(  j\right)  =\left\{
\begin{array}
[c]{cc}%
1, & j\equiv n\left(  \operatorname{mod}k\right) \\
0, & j\not \equiv n\left(  \operatorname{mod}k\right)  .
\end{array}
\right.
\]
We then find that
\[
A_{n-k}\left(  k\right)  =A_{n}\left(  k\right)  -\frac{1}{n}%
\]
and%
\[
A_{n}\left(  k\right)  =\sum\limits_{j=1}^{n}\frac{f_{n}\left(  j\right)  }%
{j}=\frac{f_{n}\left(  n\right)  }{n}=\frac{1}{n},\text{ }1\leq n\leq k.
\]
Thus,
\[
\int_{0}^{\infty}\frac{x^{s-1}e^{-kx}}{e^{-xk}-1}\log\left(  1-e^{-x}\right)
dx=\Gamma\left(  s\right)  \zeta_{A\left(  k\right)  }\left(  s\right)
-\Gamma\left(  s\right)  \zeta\left(  s+1\right)  .
\]

For the second integral we take advantage of
\[
\Gamma\left(  s\right)  \zeta\left(  s\right)  =\int_{0}^{\infty}\frac
{t^{s-1}e^{-t}}{1-e^{-t}}dt
\]
and find that
\begin{align*}
\int_{0}^{\infty}\frac{x^{s-1}e^{-kx}}{1-e^{-kx}}\log xdx  &  =\frac{1}{k^{s}%
}\int_{0}^{\infty}\frac{t^{s-1}e^{-t}}{1-e^{-t}}\log tdt-\frac{\log k}{k^{s}%
}\int_{0}^{\infty}\frac{t^{s-1}e^{-t}}{1-e^{-t}}dt\\
&  =\frac{1}{k^{s}}\Gamma^{\prime}\left(  s\right)  \zeta\left(  s\right)
+\frac{1}{k^{s}}\Gamma\left(  s\right)  \zeta^{\prime}\left(  s\right)
-\frac{\log k}{k^{s}}\Gamma\left(  s\right)  \zeta\left(  s\right)  .
\end{align*}
Combining the results gives (\ref{2}).

Since $I\left(  n,k\right)  =0$ for $n\in\mathbb{N}$, the singularities of
$\zeta_{A\left(  k\right)  }\left(  s\right)  $ arise from the poles of
$\zeta\left(  s+1\right)  $, $\psi\left(  s\right)  \zeta\left(  s\right)  $,
$\zeta^{\prime}\left(  s\right)  $ and $\zeta\left(  s\right)  $. There is a
second-order pole at $s=1$ arising from the term $\zeta^{\prime}\left(
s\right)  /k^{s}$ with
\begin{equation}
\text{\textrm{Res}}\left(  \zeta_{A\left(  k\right)  },s=1\right)
=\lim_{s\rightarrow1}\frac{d}{ds}\left\{  \left(  s-1\right)  ^{2}%
\zeta_{A\left(  k\right)  }\left(  s\right)  \right\}  =\frac{\gamma}{k},
\label{res1}%
\end{equation}
where $\mathrm{Res}\left(  f\left(  s\right)  ,s=a\right)  $ stands for the
residue of $f\left(  s\right)  $ at $s=a$.

Moreover, there are simple poles at $s=0$ and $s=1-2j$, $j\in\mathbb{N}$,
arising from $\zeta\left(  s+1\right)  $ and/or $\psi\left(  s\right)
\zeta\left(  s\right)  $ with
\begin{align*}
\text{\textrm{Res}}\left(  \zeta_{A\left(  k\right)  },s=0\right)   &
=1+\zeta\left(  0\right)  ,\\
\text{\textrm{Res}}\left(  \zeta_{A\left(  k\right)  },s=1-2j\right)   &
=k^{2j-1}\zeta\left(  1-2j\right)  =-k^{2j-1}\frac{B_{2j}}{2j}.
\end{align*}

\section{Laurent coefficients and the values $\zeta_{A\left(  k\right)
}\left(  -2n\right)  $\label{sec3}}

In this section we study the special values $\zeta_{A\left(  k\right)
}\left(  -m\right)  $ ($k$, $m\in\mathbb{N}$), and Laurent series expansions
at its poles via integral representation stated in Theorem \ref{mteo1}.

Theorem \ref{mteo1} and (\ref{res1}) make it obvious that the Laurent series
of the function $\zeta_{A\left(  k\right)  }\left(  s\right)  $ in a
neighborhood of $s=1$ can be written as
\begin{equation}
\zeta_{A\left(  k\right)  }\left(  s\right)  =\frac{1}{\left(  s-1\right)
^{2}}\frac{1}{k}+\frac{1}{s-1}\frac{\gamma}{k}+\sum\limits_{m=0}^{\infty
}\left(  -1\right)  ^{m}\frac{\gamma_{A\left(  k\right)  }\left(  m\right)
}{m!}\left(  s-1\right)  ^{m}. \label{azetaL}%
\end{equation}
Expanding both sides of (\ref{Raabe}) in a Laurent series at $s=1$ gives
\begin{align*}
&  \frac{k}{\left(  s-1\right)  ^{2}}+\frac{1}{s-1}\sum\limits_{a=1}^{k}%
\gamma\left(  0,\frac{a}{k}\right)  +\sum\limits_{n=0}^{\infty}\left(
-1\right)  ^{n}\frac{\left(  s-1\right)  ^{n}}{n!}\sum\limits_{a=1}^{k}%
\gamma_{H}\left(  n,\frac{a}{k}\right) \\
&  =\frac{k}{\left(  s-1\right)  ^{2}}+k\frac{\ln k+\gamma}{s-1}%
+k\sum\limits_{n=0}^{\infty}\left(  \frac{\ln^{n+2}k}{\left(  n+2\right)
!}+\gamma\frac{\ln^{n+1}k}{\left(  n+1\right)  !}\right)  \left(  s-1\right)
^{n}\\
&  +k^{2}\sum\limits_{n=0}^{\infty}\sum\limits_{j=0}^{n}\binom{n}{j}\left(
-1\right)  ^{j}\gamma_{A\left(  k\right)  }\left(  j\right)  \ln^{n-j}\left(
k\right)  \frac{\left(  s-1\right)  ^{n}}{n!}%
\end{align*}
upon the use of (\ref{hzetaL}) and (\ref{azetaL}). This yields the following
multiplication property for $\gamma_{H}\left(  n,b\right)  $:

\begin{corollary}
For all $n\in\mathbb{N}\cup\left\{  0\right\}  $ we have%
\begin{align}
\sum\limits_{a=1}^{k}\gamma_{H}\left(  n,\frac{a}{k}\right)   &  =\left(
-1\right)  ^{n}\left(  \frac{\ln k}{n+2}+\gamma\right)  \frac{k\ln^{n+1}%
k}{n+1}\nonumber\\
&  +k^{2}\sum\limits_{j=0}^{n}\binom{n}{j}\left(  -1\right)  ^{n-j}%
\gamma_{A\left(  k\right)  }\left(  j\right)  \ln^{n-j}\left(  k\right)  .
\label{hgRaabe}%
\end{align}

\end{corollary}

The next proposition presents an evaluation formula for the constants
$\gamma_{A\left(  k\right)  }\left(  m\right)  $.

\begin{proposition}
For $m\in\mathbb{N}\cup\left\{  0\right\}  $, we have%
\begin{align}
\gamma_{A\left(  k\right)  }\left(  m\right)   &  =\left(  -1\right)
^{m}\zeta^{\left(  m\right)  }\left(  2\right)  -\frac{1}{k}\sum
\limits_{v=0}^{m}\binom{m}{v}\frac{\psi^{\left(  v+1\right)  }\left(
1\right)  }{v+1}\left(  -1\right)  ^{v}\log^{m-v}k\nonumber\\
&  +\left(  -1\right)  ^{m}\sum\limits_{v=0}^{m}\binom{m}{v}g_{m-v}%
i_{v,k}+\left(  \psi\left(  1\right)  -\frac{m+1}{m+2}\log k\right)
\frac{\log^{m+1}k}{k\left(  m+1\right)  }\nonumber\\
&  +\frac{1}{k}\sum\limits_{v=0}^{m}\binom{m}{v}\left(  \gamma\left(
v\right)  \log k+\gamma\left(  v+1\right)  \right)  \log^{m-v}k\nonumber\\
&  -\frac{1}{k}\sum\limits_{j=0}^{m}\binom{m}{j}\left(  -1\right)  ^{j}\left(
\log^{m-j}k\right)  \sum\limits_{v=0}^{j}\binom{j}{v}\left(  -1\right)
^{v}\psi^{\left(  j-v\right)  }\left(  1\right)  \gamma\left(  v\right)  ,
\label{ev1}%
\end{align}
where%
\begin{equation}
g_{m}=\left.  \frac{d^{m}}{ds^{m}}\frac{1}{\Gamma\left(  s\right)
}\right\vert _{s=1}\text{ and }i_{m,k}=\int_{0}^{\infty}\frac{e^{-kx}}%
{e^{-kx}-1}\log\left(  \frac{1-e^{-x}}{x}\right)  \log^{m}xdx. \label{gm-im}
\end{equation}

\end{proposition}

\begin{proof}
We focus on the RHS of (\ref{2}). To expand $\Gamma(1-s)I(s,k)$ into a series
we use (\ref{10}) and write $\Gamma(1-s)I(s,k)$ as%
\[
\Gamma(1-s)I\left(  s,k\right)  =\frac{1}{\Gamma(s)}\int_{0}^{\infty}%
\frac{x^{s-1}e^{-kx}}{e^{-kx}-1}\log\left(  \frac{1-e^{-x}}{x}\right)  dx.
\]
Hence we can write
\begin{align*}
\Gamma(1-s)I\left(  s,k\right)   &  =\sum\limits_{m=0}^{\infty}\frac{g_{m}%
}{m!}\left(  s-1\right)  ^{m}\sum\limits_{m=0}^{\infty}\frac{i_{m,k}}%
{m!}\left(  s-1\right)  ^{m}\\
&  =\sum\limits_{m=0}^{\infty}\sum\limits_{v=0}^{m}\binom{m}{v}g_{m-v}%
i_{v,k}\frac{\left(  s-1\right)  ^{m}}{m!}
\end{align*}
since the integral on the RHS and the function $1/\Gamma(s)$ are analytic at $s=1$. Notwithstanding, we employ the known Laurent expansions:
\begin{align*}
\psi\left(  s\right)  \zeta\left(  s\right)   &  =\frac{\psi\left(  1\right)
}{s-1}+\sum_{n=0}^{\infty}\frac{\psi^{\left(  n+1\right)  }\left(  1\right)
}{n+1}\frac{\left(  s-1\right)  ^{n}}{n!}\\
&  \quad+\sum_{n=0}^{\infty}\sum_{k=0}^{n}\dbinom{n}{k}\left(  -1\right)
^{n-k}\gamma\left(  n-k\right)  \psi^{\left(  k\right)  }\left(  1\right)
\frac{\left(  s-1\right)  ^{n}}{n!}%
\end{align*}
and%
\[
\zeta\left(  s\right)  =\frac{1}{s-1}+\sum_{n=0}^{\infty}\frac{\left(
-1\right)  ^{n}\gamma\left(  n\right)  }{n!}\left(  s-1\right)  ^{n}.
\]
Therefore, we find from (\ref{2}) and (\ref{azetaL}) that
\begin{align*}
&  \sum\limits_{m=0}^{\infty}\frac{\left(  -1\right)  ^{m}\gamma_{A\left(
k\right)  }\left(  m\right)  }{m!}\left(  s-1\right)  ^{m}\\
&  \ =\sum\limits_{m=0}^{\infty}\left(  \sum\limits_{v=0}^{m}\binom{m}%
{v}g_{m-v}i_{v,k}+\zeta^{\left(  m\right)  }\left(  2\right)  \right)
\frac{\left(  s-1\right)  ^{m}}{m!}\\
&  +\sum\limits_{m=0}^{\infty}\left(  -1\right)  ^{m}\left(  \frac{1}{k}%
\sum\limits_{v=0}^{m}\binom{m}{v}\gamma\left(  v\right)  \log^{m-v+1}%
k-\frac{\log k-\psi\left(  1\right)  }{k}\frac{\log^{m+1}k}{m+1}\right)
\frac{\left(  s-1\right)  ^{m}}{m!}\\
&  -\frac{1}{k}\sum\limits_{m=0}^{\infty}\sum\limits_{v=0}^{m}\binom{m}%
{v}\left(  -1\right)  ^{m-v}\frac{\psi^{\left(  v+1\right)  }\left(  1\right)
\left(  \log^{m-v}k\right)  }{v+1}\frac{\left(  s-1\right)  ^{m}}{m!}\\
&  +\sum\limits_{m=0}^{\infty}\frac{\left(  -1\right)  ^{m}}{k}\left(
\frac{\log^{m+2}k}{\left(  m+2\right)  \left(  m+1\right)  }+\sum
\limits_{v=0}^{m}\binom{m}{v}\gamma\left(  v+1\right)  \log^{m-v}k\right)
\frac{\left(  s-1\right)  ^{m}}{m!}\\
&  -\frac{1}{k}\sum\limits_{m=0}^{\infty}\sum\limits_{j=0}^{m}\binom{m}%
{j}\left(  -1\right)  ^{m-j}\left(  \log^{m-j}k\right)  \sum\limits_{v=0}%
^{j}\binom{j}{v}\left(  -1\right)  ^{v}\psi^{\left(  j-v\right)  }\left(
1\right)  \gamma\left(  v\right)  \frac{\left(  s-1\right)  ^{m}}{m!},
\end{align*}
which yields the desired result.
\end{proof}

In particular, for $k=1$, (\ref{ev1}) coincides with the evaluation formula of
$\gamma_{H}\left(  m\right)  $ recorded in \cite{KDCC}:

\begin{corollary}
\label{ghm}(\cite[Proposition 13]{KDCC}) For $m\in\mathbb{N}\cup\left\{
0\right\}  $, we have%
\begin{align*}
\gamma_{H}\left(  m\right)   &  =\left(  -1\right)  ^{m}\zeta^{\left(
m\right)  }\left(  2\right)  +\gamma\left(  m+1\right)  \frac{\left(
-1\right)  ^{m}\psi^{\left(  m+1\right)  }\left(  1\right)  }{m+1}\\
&  +\left(  -1\right)  ^{m}\sum\limits_{v=0}^{m}\binom{m}{v}\left(
g_{m-v}i_{v,1}-\left(  -1\right)  ^{m-v}\psi^{\left(  m-v\right)  }\left(
1\right)  \gamma\left(  v\right)  \right)  .
\end{align*}
where $g_{m}$ and $i_{m,1}$ are given by (\ref{gm-im}).
\end{corollary}

While considering Theorem \ref{mteo1} for non-positive integer values of $s$,
we need the coefficients $D_{n,k}$ defined by the generating function%
\begin{equation}
\frac{ze^{kz}}{e^{kz}-1}\log\left(  \frac{e^{z}-1}{z}\right)  =\sum
_{n=1}^{\infty}D_{n,k}z^{n}.\label{Dnk}%
\end{equation}
The case $k=1$ is nothing but Eq. (6) of \cite{BGP1}. Since
\[
\log\left(  \frac{e^{z}-1}{z}\right)  =\sum_{n=1}^{\infty}\frac{B_{n}\left(
1\right)  }{n!n}z^{n},
\]
(cf. \cite{BGP1}) where $B_{n}\left(  x\right)  $ are the Bernoulli
polynomials defined by%
\begin{equation}
\frac{ze^{xz}}{e^{z}-1}=\sum_{n=0}^{\infty}\frac{B_{n}\left(  x\right)  }%
{n!}z^{n},\label{Bn}%
\end{equation}
$D_{n,k}$ can be written as%
\[
D_{n,k}=\frac{\left(  -1\right)  ^{n}}{n!}\sum_{j=1}^{n}\binom{n}{j}%
\frac{k^{n-1-j}}{j}B_{n-j}B_{j},
\]
where $B_{j}=B_{j}\left(  0\right)  $ are the Bernoulli numbers.

\begin{corollary}
For $m\in\mathbb{N}$, we have%
\[
\zeta_{A\left(  k\right)  }\left(  -2m\right)  =\left(  \frac{k^{2m-1}}%
{2}-\frac{1}{4m}\right)  B_{2m}.
\]

\end{corollary}

\begin{proof}
Theorem \ref{mteo1} yields to%
\begin{align*}
\zeta_{A\left(  k\right)  }\left(  -2m\right)   &  =\Gamma\left(  2m+1\right)
I\left(  -2m,k\right)  +\zeta\left(  1-2m\right)  \\
&  -k^{2m}\left(  \lim_{s\rightarrow-2m}\psi\left(  s\right)  \zeta\left(
s\right)  +\zeta^{\prime}\left(  -2m\right)  -\zeta\left(  -2m\right)  \log
k\right)  .
\end{align*}
The values $I\left(  -2m,k\right)  $ can be evaluated by using the residue
theorem with (\ref{Dnk}) as%
\begin{align*}
I\left(  -2m,k\right)   &  =\frac{1}{2\pi i}\int_{C_{\epsilon}}\frac{ze^{kz}%
}{e^{kz}-1}\log\left(  \frac{e^{z}-1}{z}\right)  z^{-2m-1}\frac{dz}{z}\\
&  =D_{2m+1,k}\\
&  =-\frac{B_{2m}B_{1}}{\left(  2m\right)  !}\left(  k^{2m-1}+\frac{1}%
{2m}\right)  ,
\end{align*}
where we have used that $B_{2j+1}=0$ for $j\geq1$. Utilizing the facts
\[
\lim_{s\rightarrow-2m}\frac{\zeta\left(  s\right)  }{s+2m}=\zeta^{\prime
}\left(  -2m\right)  \text{ and }\lim_{s\rightarrow-2m}\left(  s+2m\right)
\psi\left(  s\right)  =-1
\]
gives%
\[
\lim_{s\rightarrow-2m}\psi\left(  s\right)  \zeta\left(  s\right)
=-\zeta^{\prime}\left(  -2m\right)  .
\]
Therefore, we deduce that%
\[
\zeta_{A\left(  k\right)  }\left(  -2m\right)  =B_{2m}\left(  \frac{k^{2m-1}%
}{2}-\frac{1}{4m}\right)  ,
\]
upon the use of $\zeta\left(  -2m\right)  =0$, $m\in\mathbb{N}$.
\end{proof}

We now consider the cases $s=-m$ when $m=0$ and $m=2j-1,$\ $j\in\mathbb{N}%
,$\ in Theorem \ref{mteo1} to obtain Laurent series expansions of
$\zeta_{A\left(  k\right)  }\left(  s\right)  $ in the form\
\[
\zeta_{A\left(  k\right)  }\left(  s\right)  =\frac{a_{-1}}{s+m}%
+a_{0}+O\left(  s+m\right)  .
\]

\begin{corollary}
In a neighborhood of zero we have%
\[
\zeta_{A\left(  k\right)  }\left(  s\right)  =\frac{1}{2s}+\gamma_{A\left(
k\right)  ,0}\left(  0\right)  +O\left(  s\right)  ,
\]
where%
\[
\gamma_{A\left(  k\right)  ,0}\left(  0\right)  =\frac{1}{2k}+\frac{\gamma}%
{2}.
\]

\end{corollary}

\begin{proof}
From Theorem \ref{mteo1} and the expansions%
\begin{align*}
\zeta(s+1) &  =\frac{1}{s}+\gamma\left(  0\right)  +O\left(  s\right)  ,\text{
}\zeta(s)=\zeta(0)+\zeta^{\prime}(0)s+O\left(  s^{2}\right)  ,\\
\psi(s) &  =-\frac{1}{s}+\psi(1)+O(s)
\end{align*}
we have
\[
\zeta_{A\left(  k\right)  }\left(  s\right)  =I\left(  0,k\right)
+\frac{1+\zeta(0)}{s}+\gamma-\psi(1)\zeta(0)+O\left(  s\right)  .
\]
Using the residue theorem with (\ref{Dnk}) we obtain%
\begin{align*}
I\left(  0,k\right)   &  =\frac{1}{2\pi i}\int_{C}\frac{ze^{kz}}{e^{kz}-1}%
\log\left(  \frac{e^{z}-1}{z}\right)  z^{-1}\frac{dz}{z}\\
&  =D_{1,k}=\frac{1}{2k},
\end{align*}
which gives%
\[
\zeta_{A\left(  k\right)  }\left(  s\right)  =\frac{1+\zeta(0)}{s}+\frac
{1}{2k}+\gamma\left(  1+\zeta(0)\right)  +O\left(  s\right)  .
\]

\end{proof}

\begin{corollary}
Let $j\geq1$ be an integer. In a neighborhood of $s=1-2j$ we have
\[
\zeta_{A\left(  k\right)  }\left(  s\right)  =\frac{\zeta(1-2j)k^{2j-1}%
}{s+2j-1}+\gamma_{A\left(  k\right)  ,1-2j}\left(  0\right)  +O(s+2j-1),
\]
where%
\[
\gamma_{A\left(  k\right)  ,1-2j}\left(  0\right)  =\left\{
\begin{array}
[c]{lc}%
\dfrac{1}{24k}-\dfrac{1}{4}-k\psi(2)\zeta(-1), & j=1,\smallskip\\
\dfrac{k^{2j-1}}{4j}%
{\displaystyle\sum\limits_{v=1}^{j}}
\binom{2j}{2v}\dfrac{B_{2j-2v}B_{2v}}{vk^{2v}}-k^{2j-1}\psi(2j)\zeta(1-2j), &
j>1.
\end{array}
\right.
\]

\end{corollary}

\begin{proof}
From Theorem \ref{mteo1} and%
\[
\psi(s)=-\frac{1}{s+m}+\psi(m+1)+O(s+m)
\]
we see that%
\begin{align*}
\zeta_{A\left(  k\right)  }\left(  s\right)   &  =\Gamma\left(  2j\right)
I\left(  1-2j,k\right)  +\zeta\left(  2-2j\right)  \\
&  -\frac{1}{k^{s}}\left(  -\frac{\zeta(1-2j)}{s+2j-1}+\psi(2j)\zeta
(1-2j)-\zeta\left(  1-2j\right)  \log k\right)  +O(s+2j-1)\\
&  =\Gamma\left(  2j\right)  I\left(  1-2j,k\right)  +\zeta\left(
2-2j\right)  \\
&  -k^{2j-1}\left(  -\frac{\zeta(1-2j)}{s+2j-1}+\psi(2j)\zeta(1-2j)\right)
+O(s+2j-1).
\end{align*}
Again, from residue theorem with (\ref{Dnk}) it follows that
\begin{align*}
I\left(  1-2j,k\right)   &  =\frac{1}{2\pi i}\int_{C}\frac{ze^{kz}}{e^{kz}%
-1}\log\left(  \frac{e^{z}-1}{z}\right)  z^{-2j}\frac{dz}{z}\\
&  =D_{2j,k}=\frac{1}{\left(  2j\right)  !}\sum_{v=1}^{2j}\binom{2j}{v}%
\frac{k^{2j-1-v}}{v}B_{2j-v}B_{v}.
\end{align*}
If $j=1$, then
\[
I\left(  -1,k\right)  =\frac{1}{2}\left(  2B_{1}B_{1}+\frac{1}{2k}%
B_{2}\right)  =\frac{1}{4}+\frac{1}{24k}%
\]
and if $j>1$, then%
\[
I\left(  1-2j,k\right)  =\frac{1}{\left(  2j\right)  !}\sum_{v=1}^{j}%
\binom{2j}{2v}\frac{k^{2j-1-2v}}{2v}B_{2j-2v}B_{2v}.
\]
Combining the above results, with the use of $\zeta\left(  0\right)  =-1/2$
and $\zeta\left(  2-2j\right)  =0$ for $j>1$, we arrive at%
\[
\zeta_{A\left(  k\right)  }\left(  s\right)  =\frac{\zeta(1-2j)k^{2j-1}%
}{s+2j-1}+\gamma_{A\left(  k\right)  ,1-2j}\left(  0\right)  +O(s+2j-1).
\]

\end{proof}

\subsection{A limit representation for $\gamma_{A\left(  k\right)  }\left(
m\right)  $}

In this section we prove the following limit representation for $\gamma
_{A\left(  k\right)  }\left(  m\right)  $.

\begin{theorem}
\label{mteo2}The coefficients $\gamma_{A\left(  k\right)  }\left(  m\right)  $
can be determined by%
\begin{equation}
\gamma_{A\left(  k\right)  }\left(  m\right)  =\lim_{x\rightarrow\infty
}\left(  \sum_{n\leq x}\frac{A_{n}\left(  k\right)  \log^{m}n}{n}-\frac{1}%
{k}\frac{\log^{m+2}x}{m+2}-\frac{\gamma}{k}\frac{\log^{m+1}x}{m+1}\right)  .
\label{lim}%
\end{equation}

\end{theorem}

For the proof, we shall need the following lemmas.

\begin{lemma}
[Abel summation formula](see \cite[Theorem 4.2]{Ap}) \label{lem1}If $\left\{
b_{1},b_{2},b_{3},\ldots\right\}  $ is a sequence of complex numbers and
$v\left(  x\right)  $ is a function with a continuous derivative for
$x>\alpha$, then%
\[
\sum_{n\leq x}b_{n}v\left(  n\right)  =\left(  \sum_{n\leq x}b_{n}\right)
v\left(  x\right)  -\int\limits_{\alpha}^{x}\left(  \sum_{n\leq t}%
b_{n}\right)  v^{\prime}\left(  t\right)  dt.
\]

\end{lemma}

The following lemma follows by setting $b_{n}=A_{n}\left(  k\right)  $ and
$v\left(  x\right)  =x^{-s}$ in Lemma \ref{lem1}.

\begin{lemma}
\label{lem2}For $\operatorname{Re}\left(  s\right)  >1$ we have%
\[
\zeta_{A\left(  k\right)  }\left(  s\right)  =s\int\limits_{1}^{\infty
}t^{-s-1}\left(  \sum_{n\leq t}A_{n}\left(  k\right)  \right)  dt.
\]

\end{lemma}

\begin{lemma}
\label{lemAnk}We have%
\[
\sum_{n\leq x}A_{n}\left(  k\right)  =\frac{x}{k}\log\left(  x\right)
+\frac{x}{k}\left(  \gamma-1\right)  +C_{k}+O\left(  \log x\right)  ,
\]
where%
\begin{equation}
C_{k}=\frac{1}{2}+\frac{\gamma}{k}-\gamma-\frac{3+k}{2k}\log k+\frac{1}{k^{2}%
}\sum_{j=1}^{k}j\gamma\left(  0,\frac{j}{k}\right)  . \label{Ck}%
\end{equation}

\end{lemma}

\begin{proof}
It follows from (\ref{Ank}) that%
\begin{align*}
\sum_{n=1}^{kN}A_{n}\left(  k\right)   &  =\sum_{m=1}^{kN}\frac{1}{m}%
\sum_{n=m}^{kN}f_{n}\left(  m\right) \\
&  =\sum_{j=1}^{N}\sum_{m=\left(  j-1\right)  k+1}^{jk}\frac{1}{m}\left(
\sum_{v=1}^{N-j}\sum_{n=m+\left(  v-1\right)  k}^{m+vk-1}f_{n}\left(
m\right)  +f_{kN}\left(  m\right)  \right)  .
\end{align*}
Since $f_{n}\left(  m\right)  =\left\{
\begin{array}
[c]{cc}%
1, & n=m+\left(  v-1\right)  k\equiv m\left(  \operatorname{mod}k\right) \\
0, & m+\left(  v-1\right)  k<n\leq m+vk-1
\end{array}
\right.  $ we arrive at\
\begin{align*}
\sum_{n=1}^{kN}A_{n}\left(  k\right)   &  =\sum_{j=1}^{N}\sum_{m=\left(
j-1\right)  k+1}^{jk}\frac{1}{m}\left(  N-j+f_{kN}\left(  m\right)  \right) \\
&  =NH_{kN}-\sum_{j=1}^{N}j\sum_{m=1}^{k}\frac{1}{m+\left(  j-1\right)
k}+\frac{H_{N}}{k}.
\end{align*}
Here the double sum can be evaluated as%
\begin{align*}
\sum_{j=1}^{N}\sum_{m=1}^{k}\frac{j}{m+\left(  j-1\right)  k}  &  =H_{kN}%
+\sum_{j=0}^{N-1}\sum_{m=1}^{k}\frac{j}{m+jk}\\
&  =H_{kN}+N-\frac{1}{k}\sum_{j=0}^{N-1}\sum_{m=1}^{k}\frac{m}{m+jk}\\
&  =H_{kN}+N-\frac{1}{k}\sum_{m=1}^{k}\frac{m}{k}H_{N-1}\left(  \frac{m}%
{k}\right),
\end{align*}
which gives
\[
\sum_{n=1}^{kN}A_{n}\left(  k\right)  =NH_{kN}+\frac{H_{N}}{k}-H_{kN}%
-N+\frac{1}{k}\sum_{m=1}^{k}\frac{m}{k}H_{N-1}\left(  \frac{m}{k}\right)  .
\]
We now use the well-known identities%
\[
H_{N}=\log N+\gamma+\frac{1}{2N}+O\left(  \frac{1}{N^{2}}\right)  ,\text{ as
}N\rightarrow\infty
\]
and
\[
H_{N}\left(  a\right)  =\log\left(  N+a\right)  +\gamma\left(  0,a\right)
+O\left(  \frac{1}{N+a}\right)  ,\text{ as }N\rightarrow\infty.
\]
Hence we deduce that%
\begin{align*}
\sum_{n=1}^{kN}A_{n}\left(  k\right)   &  =N\log kN+N\left(  \gamma-1\right)
+\frac{3-k}{2k}\log\left(  N\right) \\
&  +\frac{1}{2}+\left(  \frac{1}{k}-1\right)  \gamma-\log k+\frac{1}{k}%
\sum_{m=1}^{k}\frac{m}{k}\gamma\left(  0,\frac{m}{k}\right)  +O\left(
\frac{1}{N}\right)  ,
\end{align*}
which implies
\begin{align*}
\sum_{n\leq x}A_{n}\left(  k\right)   &  =\frac{x}{k}\log\left(  x\right)
+\frac{x}{k}\left(  \gamma-1\right) \\
&  +\frac{1}{2}+\left(  \frac{1}{k}-1\right)  \gamma-\frac{3+k}{2k}\log
k+\frac{1}{k^{2}}\sum_{j=1}^{k}j\gamma\left(  0,\frac{j}{k}\right)  +O\left(
\log x\right)  .
\end{align*}

\end{proof}

According to Lemma \ref{lemAnk} we set%
\begin{align}
E\left(  x,k\right)   &  =\sum_{n\leq x}A_{n}\left(  k\right)  -\frac{x}%
{k}\log x-\frac{x}{k}\left(  \gamma-1\right)  -C_{k}\nonumber\\
&  =O\left(  \log x\right)  . \label{5}%
\end{align}

\begin{lemma}
\label{lem3}Let
\[
f\left(  s,k\right)  =s\int\limits_{1}^{\infty}t^{-s-1}E\left(  t,k\right)
dt.
\]
Then, for $\operatorname{Re}\left(  s\right)  >0$,
\[
f\left(  s,k\right)  =\frac{1}{k}-\frac{\gamma}{k}-C_{k}+\sum_{m=0}^{\infty
}\left(  -1\right)  ^{m}\frac{\gamma_{A\left(  k\right)  }\left(  m\right)
}{m!}\left(  s-1\right)  ^{m},
\]
where $C_{k}$ is given by (\ref{Ck}).
\end{lemma}

\begin{proof}
For $\operatorname{Re}\left(  s\right)  >0$ the integral is an analytic
function. Moreover, for $\operatorname{Re}\left(  s\right)  >1$, it follows
from (\ref{5}) and Lemma \ref{lem2} that%
\begin{align*}
s\int\limits_{1}^{\infty}t^{-s-1}E\left(  t,k\right)  dt  &  =\frac
{1/k}{\left(  s-1\right)  ^{2}}+\frac{\gamma/k}{s-1}-\frac{s}{k}\frac
{1}{\left(  s-1\right)  ^{2}}-\frac{\gamma-1}{k}\frac{s}{s-1}\\
&  -C_{k}+\sum_{m=0}^{\infty}\frac{\left(  -1\right)  ^{m}\gamma_{A\left(
k\right)  }\left(  m\right)  }{m!}\left(  s-1\right)  ^{m}%
\end{align*}
which gives the desired conclusion.
\end{proof}

\begin{lemma}
\label{lem4}Let $m$ be a non-negative integer and let $u<0$. Then,%
\begin{align*}
\sum_{n\leq x}A_{n}\left(  k\right)  n^{u}\log^{m}n  &  =\frac{1}{k}%
\int\limits_{1}^{x}\left(  t^{u}\log^{m+1}t+\gamma t^{u}\log^{m}t\right)
dt+\left(  -1\right)  ^{m}f^{\left(  m\right)  }\left(  -u,k\right) \\
&  +\left(  C_{k}+\frac{\gamma-1}{k}\right)  \left\{
\begin{array}
[c]{cc}%
1, & m=0\\
0, & m>0
\end{array}
\right.  +O\left(  x^{u}\log^{m+1}x\right)
\end{align*}
where $f\left(  s,k\right)  $ is the function introduced in Lemma \ref{lem3},
and $C_{k}$ is given by (\ref{Ck}).
\end{lemma}

\begin{proof}
Set $v\left(  x\right)  =x^{u}\log^{m}x$ and $b_{n}=A_{n}\left(  k\right)  $
in Lemma \ref{lem1}. Then%
\begin{equation}
\sum_{n\leq x}A_{n}\left(  k\right)  n^{u}\log^{m}n=x^{u}\left(  \log
^{m}x\right)  \sum_{n\leq x}A_{n}\left(  k\right)  -\int\limits_{1}^{x}%
\sum_{n\leq t}A_{n}\left(  k\right)  \frac{d}{dt}\left(  t^{u}\log
^{m}t\right)  dt. \label{9}%
\end{equation}
The RHS of (\ref{9}) can be written as%
\[
RHS=\frac{L_{1}}{k}+\frac{\gamma-1}{k}L_{2}-\frac{d^{m}}{du^{m}}%
u\int\limits_{1}^{x}t^{u-1}E\left(  t,k\right)  dt+C_{k}\frac{d^{m}}{du^{m}%
}1+O\left(  x^{u}\log^{m+1}x\right)
\]
upon the use of%
\[
\frac{d}{dt}\left(  t^{u}\log^{m}t\right)  =\frac{d^{m}}{du^{m}}\left(
ut^{u-1}\right)  .
\]
Here
\[
L_{1}=x^{u+1}\log^{m+1}x-\frac{d^{m}}{du^{m}}u\int\limits_{1}^{x}t^{u}\log
tdt\text{ and }L_{2}=x^{u+1}\log^{m}x-\frac{d^{m}}{du^{m}}u\int\limits_{1}%
^{x}t^{u}dt
\]
and $\dfrac{d^{m}}{du^{m}}1=\left\{
\begin{array}
[c]{cc}%
1, & m=0\\
0, & m>0
\end{array}
\right.  $. It is easily seen that
\[
L_{1}=\int\limits_{1}^{x}t^{u}\log^{m+1}tdt+\int\limits_{1}^{x}t^{u}\log
^{m}tdt\text{ and }L_{2}=\int\limits_{1}^{x}t^{u}\log^{m}tdt+\frac{d^{m}%
}{du^{m}}1.
\]
These give%
\begin{align*}
\sum_{n\leq x}A_{n}\left(  k\right)  n^{u}\log^{m}n  &  =\frac{1}{k}%
\int\limits_{1}^{x}\left(  t^{u}\log^{m+1}t+\gamma t^{u}\log^{m}t\right)
dt+\left(  -1\right)  ^{m}f^{\left(  m\right)  }\left(  -u,k\right) \\
&  +\left(  \frac{\gamma-1}{k}+C_{k}\right)  \frac{d^{m}}{du^{m}}1+O\left(
x^{u}\log^{m+1}x\right)  ,
\end{align*}
the desired result.
\end{proof}

We are now ready to prove Theorem \ref{mteo2}.

\begin{proof}
[Proof of Theorem \ref{mteo2}]From Lemma \ref{lem3} we have%
\begin{equation}
f^{\left(  m\right)  }\left(  1,k\right)  =\left(  -1\right)  ^{m}%
\gamma_{A\left(  k\right)  }\left(  m\right)  \text{ for }m>0\text{ and
}f\left(  1,k\right)  =\frac{1-\gamma}{k}-C_{k}+\gamma_{A\left(  k\right)
}\left(  0\right)  . \label{14}%
\end{equation}
Setting $u=-1$ in Lemma \ref{lem4} yields to%
\begin{align}
\sum_{n\leq x}\frac{A_{n}\left(  k\right)  \log^{m}n}{n}  &  =\frac{1}{k}%
\frac{\log^{m+2}x}{m+2}+\frac{\gamma}{k}\frac{\log^{m+1}x}{m+1}+\left(
-1\right)  ^{m}f^{\left(  m\right)  }\left(  1,k\right) \nonumber\\
&  +\left(  C_{k}+\frac{\gamma-1}{k}\right)  \frac{d^{m}}{du^{m}}1+O\left(
x^{u}\log^{m+1}x\right)  . \label{15}%
\end{align}
Hence (\ref{14}) and (\ref{15}) lead to%
\[
\gamma_{A\left(  k\right)  }\left(  0\right)  =\lim_{x\rightarrow\infty
}\left(  \sum_{n\leq x}\frac{A_{n}\left(  k\right)  }{n}-\frac{1}{k}\frac
{\log^{2}x}{2}-\frac{\gamma}{k}\log x\right)
\]
and
\[
\gamma_{A\left(  k\right)  }\left(  m\right)  =\lim_{x\rightarrow\infty
}\left(  \sum_{n\leq x}\frac{A_{n}\left(  k\right)  \log^{m}n}{n}-\frac{1}%
{k}\frac{\log^{m+2}x}{m+2}-\frac{\gamma}{k}\frac{\log^{m+1}x}{m+1}\right)  .
\]

\end{proof}

\section{Two evaluation formulas for $\gamma_{H}\left(  n,1/2\right)
$\label{sec5}}

In \cite{KDCC}, we have presented a formula for the harmonic
Stieltjes constant $\gamma_{H}\left(  m,1/2\right)  $. Here we offer two other
evaluation formulas. The first one is a consequence of (\ref{hgRaabe}):

\begin{corollary}
\label{ghm1b2-1}For a non-negative integer $m$ we have%
\begin{align}
\gamma_{H}\left(  m,1/2\right)   &  =-\gamma_{H}\left(  m\right)  +\left(
-1\right)  ^{m}\left(  \frac{\ln2}{n+2}+\gamma\right)  \frac{2\ln^{m+1}2}%
{m+1}\nonumber\\
&  +4\sum\limits_{j=0}^{m}\binom{m}{j}\left(  -1\right)  ^{m-j}\gamma
_{A\left(  2\right)  }\left(  j\right)  \ln^{m-j}\left(  2\right)  .
\label{ev2a}%
\end{align}

\end{corollary}

The second one is as follows:

\begin{proposition}
\label{ghm1b2-2}For a non-negative integer $m$ we have%
\begin{align}
\gamma_{H}\left(  m,1/2\right)   &  =-\gamma_{H}\left(  m\right)  +\left(
-1\right)  ^{m}\left(  \frac{\log2}{m+2}+\gamma\right)  \frac{2\log^{m+1}%
2}{\left(  m+1\right)  }\nonumber\\
&  +2\left(  -1\right)  ^{m}\sum_{j=0}^{m}\binom{m}{j}\left(  \left(
-1\right)  ^{j}\gamma_{H}\left(  j\right)  -\sum\limits_{v=0}^{j}\binom{j}%
{v}g_{j-v}J_{v}\right)  \left(  \log^{m-j}2\right)  , \label{ev2b}%
\end{align}
where $g_{m}$ is given by (\ref{gm-im}) and
\[
J_{m}=\int_{0}^{\infty}\frac{\log\left(  1-e^{-x}\right)  }{1+e^{-x}}\log
^{m}xdx.
\]

\end{proposition}

\begin{proof}
We exploit the equation%
\[
\zeta_{H}\left(  s,a\right)  =\frac{1}{\Gamma\left(  s\right)  }\int%
_{0}^{\infty}\frac{x^{s-1}}{1-e^{-x}}e^{-xa}\Phi\left(  e^{-x},1;a\right)  dx
\]
(\cite[Eq. (14)]{KDCC}) as%
\begin{align*}
&  \zeta_{H}\left(  s,a+1\right)  +\zeta_{H}\left(  s,a+\frac{1}{2}\right)  \\
&  =\frac{1}{\Gamma\left(  s\right)  }\int_{0}^{\infty}\frac{x^{s-1}}%
{1-e^{-x}}\left(  e^{-x\left(  a+1\right)  }\Phi\left(  e^{-x},1;a+1\right)
+e^{-x\left(  a+1/2\right)  }\Phi\left(  e^{-x},1;a+1/2\right)  \right)  dx.
\end{align*}
The equality%
\[
\frac{\Phi\left(  e^{-x},1;a+1\right)  }{e^{x\left(  a+1\right)  }}+\frac
{\Phi\left(  e^{-x},1;a+1/2\right)  }{e^{x\left(  a+1/2\right)  }}=2\frac
{\Phi\left(  e^{-x/2},1;2a\right)  }{e^{2a\left(  x/2\right)  }}-\frac
{e^{-ax}}{a}%
\]
gives rise to%
\begin{align*}
&  \zeta_{H}\left(  s,a+1\right)  +\zeta_{H}\left(  s,a+\frac{1}{2}\right)  \\
&  =\frac{2^{s}}{\Gamma\left(  s\right)  }\int_{0}^{\infty}\frac{\left(
x/2\right)  ^{s-1}}{1-e^{-x/2}}e^{-2a\left(  x/2\right)  }\Phi\left(
e^{-x/2},1;2a\right)  d\left(  \frac{x}{2}\right)  \\
&  +\frac{2^{s}}{\Gamma\left(  s\right)  }\int_{0}^{\infty}\frac{\left(
x/2\right)  ^{s-1}e^{-2a\left(  x/2\right)  }}{1+e^{-x/2}}\Phi\left(
e^{-x/2},1;2a\right)  d\left(  \frac{x}{2}\right)  -\frac{1}{a\Gamma\left(
s\right)  }\int_{0}^{\infty}\frac{x^{s-1}e^{-ax}}{1-e^{-x}}dx.
\end{align*}

It is obvious from the generating function
\[
\sum_{n=0}^{\infty}H_{n}^{^{-}}\left(  a\right)  t^{n}=\frac{\Phi\left(
-t,1;a\right)  }{1-t}\text{, where }H_{n}^{^{-}}\left(  a\right)  =\sum
_{j=0}^{n}\frac{\left(  -1\right)  ^{j}}{j+a}\text{,}%
\]
that%
\begin{equation}
\frac{1}{\Gamma\left(  s\right)  }\int_{0}^{\infty}\frac{x^{s-1}e^{-ax}%
}{1+e^{-x}}\Phi\left(  e^{-x},1;a\right)  dx=\sum_{n=0}^{\infty}\frac{\left(
-1\right)  ^{n}H_{n}^{-}\left(  a\right)  }{\left(  n+a\right)  ^{s}}%
=:\eta_{H^{-}}\left(  s,a\right)  . \label{11}%
\end{equation}
Thus we arrive at
\begin{equation}
\zeta_{H}\left(  s,a+1\right)  +\zeta_{H}\left(  s,a+\frac{1}{2}\right)
=2^{s}\left\{  \zeta_{H}\left(  s,2a\right)  +\eta_{H^{-}}\left(  s,2a\right)
\right\}  -\frac{1}{a}\zeta\left(  s,a\right)  . \label{12}%
\end{equation}
Since%
\begin{equation}
\zeta_{H}\left(  s,a+1\right)  =\zeta_{H}\left(  s,a\right)  -\frac{1}{a}%
\zeta\left(  s,a\right)  \label{16}%
\end{equation}
(\ref{12}) becomes%
\[
\zeta_{H}\left(  s,a\right)  +\zeta_{H}\left(  s,a+\frac{1}{2}\right)
=2^{s}\left\{  \zeta_{H}\left(  s,2a\right)  +\eta_{H^{-}}\left(  s,2a\right)
\right\}  .
\]
In particular, for $a=1/2$, we have%
\begin{equation}
\zeta_{H}\left(  s,\frac{1}{2}\right)  =\left(  2^{s}-1\right)  \zeta
_{H}\left(  s\right)  +2^{s}\eta_{H^{-}}\left(  s\right)  , \label{3}%
\end{equation}
where we have used that $\zeta_{H}\left(  s,1\right)  =\zeta_{H}\left(
s\right)  $ and $\eta_{H^{-}}\left(  s,1\right)  =\eta_{H^{-}}\left(
s\right)  $. Making use of (\ref{hzetaL}) and
\[
\eta_{H^{-}}\left(  s\right)  =\sum_{m=0}^{\infty}\frac{\eta_{H^{-}}^{\left(
m\right)  }\left(  1\right)  }{m!}\left(  s-1\right)  ^{m}\text{ }%
\]
we deduce that
\begin{align*}
\sum_{m=0}^{\infty}\left(  -1\right)  ^{m}\frac{\gamma_{H}\left(
m,1/2\right)  }{m!}\left(  s-1\right)  ^{m} &  =2\sum_{m=0}^{\infty}\left(  \frac{\log^{m+2}2}{\left(  m+2\right)
!}+\gamma\frac{\log^{m+1}2}{\left(  m+1\right)  !}-\left(  -1\right)
^{m}\frac{\gamma_{H}\left(  m\right)  }{m!}\right)  \left(  s-1\right)  ^{m}\\
&  +2\sum_{m=0}^{\infty}\sum_{j=0}^{m}\binom{m}{j}\left\{  \left(  -1\right)
^{j}\gamma_{H}\left(  j\right)  -\eta_{H^{-}}^{\left(  j\right)  }\left(
1\right)  \left(  \log^{m-j}2\right)  \right\}  \frac{\left(  s-1\right)
^{m}}{m!},
\end{align*}
which entails
\begin{align*}
\gamma_{H}\left(  m,1/2\right)   &  =-\gamma_{H}\left(  m\right)  +\left(
-1\right)  ^{m}\left(  \frac{\log2}{m+2}+\gamma\right)  \frac{2\log^{m+1}%
2}{\left(  m+1\right)  }\\
&  +2\left(  -1\right)  ^{m}\sum_{j=0}^{m}\binom{m}{j}\left(  \left(
-1\right)  ^{j}\gamma_{H}\left(  j\right)  +\eta_{H^{-}}^{\left(  j\right)
}\left(  1\right)  \right)  \left(  \log^{m-j}2\right)  .
\end{align*}
The assertion%
\[
\eta_{H^{-}}^{\left(  j\right)  }\left(  1\right)  =-\sum\limits_{v=0}%
^{j}\binom{j}{v}g_{j-v}J_{v}%
\]
follows from (\ref{11}).
\end{proof}

Recall the formula recorded in \cite[Proposition 12]{KDCC}:%
\begin{align}
\gamma_{H}\left(  m,1/2\right)    =\gamma_{H}\left(  m\right)  +2\left(
-1\right)  ^{m}\frac{\log^{m+1}2}{m+1}
 +2\sum_{v=0}^{m}\binom{m}{v}\left(  -\log2\right)  ^{m-v}\left(
\gamma_{\widetilde{H}}\left(  v\right)  +\gamma_{H^{^{-}}}\left(  v\right)
\right) ,\label{8a}%
\end{align}
where $\gamma_{\widetilde{H}}\left(  m\right)  $ and $\gamma_{H^{^{-}}}\left(
m\right)  $ are defined by
\[%
\begin{array}
[c]{c}%
\eta_{H}\left(  s\right)  =%
{\displaystyle\sum\limits_{n=1}^{\infty}}
\left(  -1\right)  ^{n-1}\dfrac{H_{n}}{n^{s}}=%
{\displaystyle\sum\limits_{m=0}^{\infty}}
\left(  -1\right)  ^{m}\dfrac{\gamma_{\widetilde{H}}\left(  m\right)  }%
{m!}\left(  s-1\right)  ^{m},\medskip\\
\zeta_{H^{-}}\left(  s\right)  =%
{\displaystyle\sum\limits_{n=1}^{\infty}}
\dfrac{H_{n}^{^{-}}}{n^{s}}=\dfrac{\log2}{s-1}+%
{\displaystyle\sum\limits_{m=0}^{\infty}}
\left(  -1\right)  ^{m}\dfrac{\gamma_{H^{^{-}}}\left(  m\right)  }{m!}\left(
s-1\right)  ^{m}.
\end{array}
\]
Employing the binomial transform, formula (\ref{8a}) can equivalently be
written as\textbf{ }%
\begin{align}
\gamma_{H^{^{-}}}\left(  m\right)   &  =-\gamma_{\widetilde{H}}\left(
m\right)  +\frac{1}{2}\sum_{v=0}^{m}\binom{m}{v}\left(  \log2\right)
^{m-v}\nonumber\\
&  \qquad\times\left(  \gamma_{H}\left(  v,1/2\right)  -\gamma_{H}\left(
v\right)  -2\left(  -1\right)  ^{v}\frac{\log^{v+1}2}{v+1}\right)  .\label{8b}%
\end{align}
Since the constants that occur on the RHS are computable ($\gamma_{H}\left(
k,1/2\right)  $ by Corollary \ref{ghm1b2-1}, $\gamma_{H}\left(  k\right)  $ by
Corollary \ref{ghm}, $\gamma_{\widetilde{H}}\left(  m\right)  $ by
\cite[Proposition 14]{KDCC}), (\ref{8b}) can be used to calculate
$\gamma_{H^{^{-}}}\left(  m\right)  $.

\section{Proof of Theorem \ref{mteo3} and its consequences\label{sec6}}

In this section, we first prove Theorem \ref{mteo3} and then present some
results for the harmonic zeta functions $\zeta_{O}\left(  s\right)  $ and
$S_{1,s}^{++}\left(  0,1/2\right)  $.

\begin{proof}
[Proof of Theorem \ref{mteo3}]Let $\operatorname{Re}\left(  s\right)  >1$.
Then,%
\begin{align*}
\eta_{H}\left(  s,a\right)   =\sum_{n=0}^{\infty}\frac{\left(  -1\right)
^{n}H_{n}\left(  a\right)  }{\left(  n+a\right)  ^{s}}
  =2^{1-s}\sum_{n=1}^{\infty}\frac{H_{2n}\left(  a\right)  }{\left(
n+a/2\right)  ^{s}}+\frac{2}{a^{s+1}}-\zeta_{H}\left(  s,a\right)  .
\end{align*}
We now use the facts
\[
H_{2n}\left(  a\right)  =\mathcal{H}_{2n}\left(  a\right)  +\frac{1}{a}%
\]
and
\[
\sum_{n=1}^{\infty}\frac{\mathcal{H}_{2n}\left(  a\right)  }{\left(
n+b\right)  ^{s}}=\frac{1}{2}S_{1,s}^{++}\left(  \frac{a-1}{2},b\right)
+\frac{1}{2}S_{1,s}^{++}\left(  \frac{a}{2},b\right)
\]
(see \cite[Eq. (2.19)]{SoC}). We then find that
\[
\eta_{H}\left(  s,a\right)  =2^{-s}S_{1,s}^{++}\left(  \frac{a-1}{2},\frac
{a}{2}\right)  +2^{-s}S_{1,s}^{++}\left(  \frac{a}{2},\frac{a}{2}\right)
+\frac{2^{1-s}}{a}\zeta\left(  s,\frac{a}{2}\right)  -\zeta_{H}\left(
s,a\right)
\]
The equality
\begin{equation}
S_{1,s}^{++}\left(  a,b\right)  =\sum_{n=1}^{\infty}\frac{\mathcal{H}%
_{n}\left(  a\right)  }{\left(  n+b\right)  ^{s}}=\sum_{n=0}^{\infty}%
\frac{H_{n}\left(  a+1\right)  }{\left(  n+b+1\right)  ^{s}}=\zeta_{H}\left(
s,a+1,b+1\right)  \label{17}%
\end{equation}
and (\ref{16}) yield%
\begin{equation}
\zeta_{H}\left(  s,\frac{a+1}{2},\frac{a+1}{2}+\frac{1}{2}\right)
=2^{s}\left(  \eta_{H}\left(  s,a\right)  +\zeta_{H}\left(  s,a\right)
\right)  -\zeta_{H}\left(  s,\frac{a}{2}\right)  . \label{18}%
\end{equation}

Then the assertion (\ref{19}) follows from (\ref{18}) with $a=1$ since
$\zeta_{H}\left(  s,1\right)  =\zeta_{H}\left(  s\right)  $, $\eta_{H}\left(
s,1\right)  =\eta_{H}\left(  s\right)  $ and%
\[
\zeta_{H}\left(  s,1,\frac{1}{2}+1\right)  =\sum_{n=0}^{\infty}\frac
{H_{n}\left(  1\right)  }{\left(  n+\frac{1}{2}+1\right)  ^{s}}=\sum
_{n=1}^{\infty}\frac{H_{n}}{\left(  n+\frac{1}{2}\right)  ^{s}}=S_{1,s}%
^{++}\left(  0,\frac{1}{2}\right)  .
\]

For (\ref{ozeta}) we set $a=2$ in (\ref{18}). The LHS becomes%
\begin{align*}
\zeta_{H}\left(  s,1+\frac{1}{2},1+1\right)   &  =S_{1,s}^{++}\left(  \frac
{1}{2},1\right)  =\sum_{n=1}^{\infty}\frac{\mathcal{H}_{n}\left(  \frac{1}%
{2}\right)  }{\left(  n+1\right)  ^{s}}\\
&  =2\sum_{n=2}^{\infty}\frac{1}{n^{s}}\sum_{j=2}^{n}\frac{1}{2j-1}%
=2\sum_{n=1}^{\infty}\frac{O_{n}}{n^{s}}-2\zeta\left(  s\right)  ,
\end{align*}
upon the use of (\ref{17}). While the RHS becomes%
\[
2^{s}\left(  \zeta_{H}\left(  s\right)  -\eta_{H}\left(  s\right)
+\eta\left(  s\right)  -\zeta\left(  s\right)  \right)  -\zeta_{H}\left(
s\right)
\]
since
\[
\zeta_{H}\left(  s,2\right)  \overset{\text{(\ref{16})}}{=}\zeta_{H}\left(  s\right)
-\zeta\left(  s\right)
\text{ and }
\eta_{H}\left(  s,2\right) =\eta\left(  s\right)
-\eta_{H}\left(  s\right)  .
\]
Hence%
\begin{align*}
S_{1,s}^{++}\left(  \frac{1}{2},1\right)   =2\sum_{n=1}^{\infty}\frac
{O_{n}}{n^{s}}-2\zeta\left(  s\right) =2^{s}\left(  \zeta_{H}\left(  s\right)  -\eta_{H}\left(  s\right)
+\eta\left(  s\right)  -\zeta\left(  s\right)  \right)  -\zeta_{H}\left(
s\right)
\end{align*}
or%
\[
\zeta_{O}\left(  s\right)  =2^{s-1}\left(  \zeta_{H}\left(  s\right)
-\eta_{H}\left(  s\right)  \right)  -\frac{1}{2}\zeta_{H}\left(  s\right)  ,
\]
upon the use of $\eta\left(  s\right)  =\left(  1-2^{1-s}\right)  \zeta\left(
s\right)  $.
\end{proof}

We now order some consequences of Theorem \ref{mteo3}. By appealing the
Laurent/Taylor expansions of $\zeta_{H}\left(  s\right)  $, $\eta_{H}\left(
s\right)  $ and $\zeta_{H}\left(  s,1/2\right)  $ in a neighborhood of $s=1$,
one can obtain the following Laurent expansions:

\begin{corollary}
\label{ozetaL}In a neighborhood of $s=1$ we have%
\[
\zeta_{O}\left(  s\right)  =\frac{1}{2}\frac{1}{\left(  s-1\right)  ^{2}%
}+\frac{\log2+\gamma/2}{\left(  s-1\right)  }+\sum\limits_{n=0}^{\infty
}\left(  -1\right)  ^{n}\frac{\gamma_{O}\left(  n\right)  }{n!}\left(
s-1\right)  ^{n},
\]
where the coefficients $\gamma_{O}\left(  n\right)  $ can be determined by%
\begin{align}
\gamma_{O}\left(  n\right)   &  =\left(  -1\right)  ^{n}\left(  \frac{\log
2}{n+2}+\gamma\right)  \frac{\log^{n+1}2}{n+1}-\frac{1}{2}\gamma_{H}\left(
n\right)  \nonumber\\
&  +\left(  -1\right)  ^{n}\sum_{v=0}^{n}\binom{n}{v}\left(  \left(
-1\right)  ^{v}\gamma_{H}\left(  v\right)  -\eta_{H}^{\left(  v\right)
}\left(  1\right)  \right)  \log^{n-v}2.\label{gammao}%
\end{align}

\end{corollary}

\begin{corollary}
\label{dnL}In a neighborhood of $s=1$ we have%
\[
S_{1,s}^{++}\left(  0,\frac{1}{2}\right)  =\frac{1}{\left(  s-1\right)  ^{2}%
}+\frac{\gamma}{s-1}+\sum\limits_{n=0}^{\infty}\left(  -1\right)  ^{n}%
\frac{d_{n}}{n!}\left(  s-1\right)  ^{n},
\]
where the coefficients $d_{n}$ can be determined by%
\begin{align}
d_{n} &  =-\gamma_{H}\left(  n,1/2\right)  +2\left(  -1\right)  ^{n}\left(
\frac{\log2}{n+2}+\gamma\right)  \frac{\log^{n+1}2}{n+1}
\nonumber\\
&  +2\left(  -1\right)  ^{n}\sum_{v=0}^{n}\binom{n}{v}\left(  \left(
-1\right)  ^{v}\gamma_{H}\left(  v\right)  +\eta_{H}^{\left(  v\right)
}\left(  1\right)  \right)  \log^{n-v}2.\label{dn}%
\end{align}

\end{corollary}

Moreover, we observe from (\ref{19}) and (\ref{ozeta}) that the functions
$\zeta_{O}\left(  s\right)  $ and $S_{1,s}^{++}\left(  0,1/2\right)  $ have
simple poles at $s=1-2m$, $m\in\mathbb{N}$. Since
\[
\mathrm{Res}\left(  \zeta_{H}\left(  s\right)  ,s=1-2m\right)  =\mathrm{Res}%
\left(  \zeta_{H}\left(  s,1/2\right)  ,s=1-2m\right)  =\zeta\left(
1-2m\right)  ,
\]
we find that
\[
\mathrm{Res}\left(  S_{1,s}^{++}\left(  0,\frac{1}{2}\right)  ,s=1-2m\right)
=\left(  2^{1-2m}-1\right)  \zeta\left(  1-2m\right)
\]
and
\[
\mathrm{Res}\left(  \zeta_{O}\left(  s\right)  ,s=1-2m\right)  =\frac{1}%
{2}\left(  2^{1-2m}-1\right)  \zeta\left(  1-2m\right)  .
\]

Furthermore, we may evaluate the values at $s=-2m$, $m\in\mathbb{N\cup
}\left\{  0\right\}  $, of $\zeta_{O}\left(  s\right)  $ and $S_{1,s}%
^{++}\left(  0,1/2\right)  $.

\begin{corollary}
\label{cor0}We have%
\[
S_{1,0}^{++}\left(  0,\frac{1}{2}\right)  =1\text{ and }\zeta_{O}\left(
0\right)  =0.
\]

\end{corollary}

\begin{proof}
The proof follows by evaluating the following limits:
\[
\lim_{s\rightarrow0}S_{1,s}^{++}\left(  0,\frac{1}{2}\right)  =\lim
_{s\rightarrow0}\left(  2^{s}\zeta_{H}\left(  s\right)  -\zeta_{H}\left(
s,\frac{1}{2}\right)  \right)  +\lim_{s\rightarrow0}2^{s}\eta_{H}\left(
s\right)
\]
and%
\[
\lim_{s\rightarrow0}\zeta_{O}\left(  s\right)  =\lim_{s\rightarrow0}\frac
{1}{2}\left(  2^{s}-1\right)  \zeta_{H}\left(  s\right)  -\lim_{s\rightarrow
0}2^{s-1}\eta_{H}\left(  s\right)  ,
\]
by utilizing that $\zeta\left(  0,1/2\right)  =1/2$, $\psi\left(  1/2\right)
=-2\log2-\gamma$ and $\eta_{H}\left(  0\right)  =\left(  \log2\right)  /2.$
\end{proof}

\begin{corollary}
For $m\in\mathbb{N}$ we have%
\[
S_{1,-2m}^{++}\left(  0,\frac{1}{2}\right)  =\left(  1-2^{2m-1}\right)
\frac{B_{2m}}{2^{2m}}%
\]
and%
\[
\zeta_{O}\left(  -2m\right)  =0.
\]

\end{corollary}

\begin{proof}
The proof of the first assertion follows from (\ref{19}), \cite[p. 400]{Ma}
(or \cite[Eq. (16)]{BGP1}), \cite[Eq. (22)]{BGP2}, \cite[p. 10]{KDCC} and the
facts
\[
\zeta\left(  -m,a\right)  =-\frac{B_{m+1}\left(  a\right)  }{m+1}\text{ and
}B_{m}\left(  \frac{1}{2}\right)  =\left(  \frac{1}{2^{m-1}}-1\right)
B_{m},\text{ }m\geq0.
\]

The proof of the second assertion follows from (\ref{ozeta}) with \cite[p.
400]{Ma} (or \cite[Eq. (16)]{BGP1}) and \cite[Eq. (22)]{BGP2}.
\end{proof}

It is worth mentioning that Corollaries \ref{ozetaL}, \ref{dnL} and \ref{cor0}
give rise to
\[
\sum\limits_{n=0}^{\infty}\frac{d_{n}}{n!}=\gamma\text{  and  }\sum
\limits_{n=0}^{\infty}\frac{\gamma_{O}\left(  n\right)  }{n!}=\frac{\gamma
-1}{2}+\log2.
\]
Further, Corollaries \ref{ozetaL} and \ref{dnL} with $\zeta_{O}\left(
2\right)  =7\zeta\left(  3\right)  /4$ (cf. \cite[Eq. (8)]{J}) and
$S_{1,2}^{++}\left(  0,1/2\right)  =7\zeta\left(  3\right)  -\pi^{2}\ln2$ (cf.
\cite[Theorem 1]{SoCv}) entail
\begin{align*}
\sum\limits_{n=0}^{\infty}\left(  -1\right)  ^{n}\frac{\gamma_{O}\left(
n\right)  }{n!} &  =\frac{7}{4}\zeta\left(  3\right)  -\frac{1+\gamma}{2}%
-\log2,\\
\sum\limits_{n=0}^{\infty}\left(  -1\right)  ^{n}\frac{d_{n}}{n!} &
=7\zeta\left(  3\right)  -\gamma-\pi^{2}\ln2-1.
\end{align*}

\end{document}